\documentclass[12pt]{amsart}
\usepackage{fullpage}
\usepackage{graphicx}
\usepackage{todonotes}
\usepackage{tkz-euclide}

\usepackage{tikz}
\usetikzlibrary{decorations.markings}
\usepackage{tkz-euclide}
\usetkzobj{all}

\newtheorem{theorem}{Theorem}[section]

\newtheorem{lemma}[theorem]{Lemma}

\newtheorem{remark}[theorem]{Remark}
\newtheorem*{thmnonnum}{Main Theorem}

\numberwithin{equation}{section}

\newcommand\CC {{\mathbb C}}
\newcommand\DD {{\mathbb D}}

\newcommand\HH {{\mathbb H}}

\newcommand\NN {{\mathbb N}}

\newcommand\QQ {{\mathbb Q}}
\newcommand\RR {{\mathbb R}}

\newcommand\ZZ {{\mathbb Z}}

\newcommand\sltwor{{\rm SL(2,\RR)}}
\newcommand\sltwoz{{\rm SL(2,\ZZ)}}
\newcommand\sltwoc{{\rm SL(2,\CC)}}

\newcommand\glgroup{{\rm GL}}
\newcommand\sugroup{{\rm SU}}

\newcommand\diameter{{\rm Diam }}

\newcommand\id{{\rm Id}}
\newcommand\im{{\rm Im }}

\newcommand\re{{\rm Re }}

\newcommand\cA{{\mathcal{A}  }}
\newcommand\cB{{\mathcal{B}  }}

\newcommand\cP{{\mathcal{P}  }}

\newcommand\cS{{\mathcal{S}  }}

\newcommand\cW{{\mathcal{W}  }}

\begin{document}



\title{On good approximations and Bowen-Series expansion}

\author{Luca Marchese}

\address{Dipartimento di Matematica, Universit\`a di Bologna, Piazza di Porta San Donato 5, 40126, Bologna, Italia}

\email{luca.marchese4@unibo.it}


\begin{abstract}
We consider the continued fraction expansion of real numbers under the action of a non-uniform lattice in $\text{PSL}(2,\RR)$ and prove metric relations between the convergents and a natural geometric notion of good approximations.
\end{abstract}

\maketitle


\section{Introduction}

Let 
$\HH:=\{z\in\CC\,;\,\im(z)>0\}$ be the \emph{upper half plane} and for $p/q\in\QQ$ let 
$H_{p/q}\subset\HH$ be the circle of diameter $1/q^2$ tangent at $p/q$. Set 
$
H_\infty=\{z\in\HH:\im(z)>1\}
$ 
and consider the family 
$
\{H_{p/q}:p/q\in\QQ\cup\{\infty\}\}
$ 
of \emph{Ford circles}, which are the orbit of $H_\infty$ under the projective action of the \emph{modular group} $\sltwoz$, that is the group of $2\times 2$ matrices with coefficients $a,b,c,d$ in $\ZZ$ (notation refers to Equation~\eqref{EquationCoefficientsSL(2,C)} below). Any two circles are either disjoint or tangent, and Figure~\ref{FigureFordCircles} shows that for any irrational $\alpha$ there exist infinitely many 
$p/q\in\QQ$ with $\alpha\in\Pi(H_{p/q})$, that is $|\alpha-p/q|<(1/2)q^{-2}$, where $\Pi(x+iy):=x$.
\begin{figure}
\begin{center}
\begin{tikzpicture}[scale=2.4]

\tikzset
{->-/.style={decoration={markings,mark=at position .2 with {\arrow{>}}},postaction={decorate}}}

\begin{scope}

\clip(-0.9,-0.5) rectangle (1.9,1.5);

\draw (-1,0) -- (2,0);

\filldraw[fill=black!5!white] (0,1/2) circle (1/2);
\draw (0,-1/30) -- (0,1/30); \node at (0,-5/30) {$\frac 01$};
\filldraw[fill=black!5!white] (1,1/2) circle (1/2);
\draw (1,-1/30) -- (1,1/30); \node at (1,-5/30) {$\frac 11$};
\filldraw[fill=black!5!white] (-1,1/2) circle (1/2);
\filldraw[fill=black!5!white] (2,1/2) circle (1/2);

\filldraw[fill=black!5!white] (1/2,1/8) circle (1/8);
\draw (1/2,-1/30) -- (1/2,1/30); \node at (1/2,-5/30) {$\frac 12$};
\filldraw[fill=black!5!white] (1/3,1/18) circle (1/18);
\draw (1/3,-1/30) -- (1/3,1/30); \node at (1/3,-5/30) {$\frac 13$};
\filldraw[fill=black!5!white] (1/4,1/32) circle (1/32);
\filldraw[fill=black!5!white] (2/3,1/18) circle (1/18);
\draw (2/3,-1/30) -- (2/3,1/30); \node at (2/3,-5/30) {$\frac 23$};
\filldraw[fill=black!5!white] (3/4,1/32) circle (1/32);

\filldraw[fill=black!5!white] (-1/2,1/8) circle (1/8);
\filldraw[fill=black!5!white] (-1/3,1/18) circle (1/18);
\filldraw[fill=black!5!white] (-2/3,1/18) circle (1/18);
\filldraw[fill=black!5!white] (-3/4,1/36) circle (1/36);
\filldraw[fill=black!5!white] (-1/4,1/36) circle (1/36);

\filldraw[fill=black!5!white] (2-1/2,1/8) circle (1/8);
\filldraw[fill=black!5!white] (2-1/3,1/18) circle (1/18);
\filldraw[fill=black!5!white] (2-2/3,1/18) circle (1/18);
\filldraw[fill=black!5!white] (2-3/4,1/36) circle (1/36);
\filldraw[fill=black!5!white] (2-1/4,1/36) circle (1/36);

\filldraw[fill=black!5!white] 
(-1,1) -- (2,1) -- (2,3) -- (-1,3) -- (-1,1);

\node at (0.5,1.3) {$H_\infty$};


\end{scope}

\end{tikzpicture}
\begin{tikzpicture}[scale=2.4]

\begin{scope}

\clip(-0.6,-0.5) rectangle (0.4,1.5);

\draw (-1.5,0) -- (1,0);

\filldraw[fill=black!5!white] (0,1/3) circle (1/3);
\node at (0.05,0.5) {$H_{p/q}$};
\draw (0,-1/30) -- (0,1/30); \node at (0,-5/30) {$\frac pq$};

\draw[-,dashed] (-0.2,3) -- (-0.2,0);

\node at (-0.2,-2/30) {$\alpha$};

\draw[<->,very thin] (-0.4,0) -- (-0.4,2/3);
\node at (-0.5,1/3) {$\frac{1}{q^2}$};

\end{scope}

\end{tikzpicture}
\end{center}
\caption{Balls $G(H_k)$, $k\in\ZZ$, tangent to $H_{p/q}=G(H_\infty)$, where $p/q=G\cdot\infty$.}
\label{FigureFordCircles}
\end{figure}
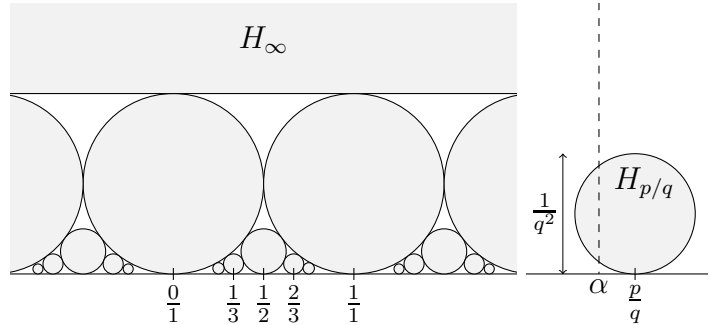
This defines the sequence of \emph{geometric good approximations} of $\alpha$ as the sequence of  
$p_n/q_n$ in $\QQ$ with $\alpha\in\Pi(B_{p_n/q_n})$. The same sequence arises from the continued fraction expansion 
$
\alpha=a_0+[a_1,a_2,\dots]
$ 
of $\alpha$, indeed the \emph{convergents} 
$
p_n/q_n:=a_0+[a_1,\dots,a_n]
$ 
satisfy:
\begin{equation}
\label{Equation(1/2)ApproximationsAreConvergents}
|\alpha-p/q|<(1/2)q^{-2}
\Rightarrow
p/q=p_n/q_n 
\text{ for some }
n\geq1.
\end{equation}
The first $n+1$ \emph{partial quotients} $a_1,\dots,a_{n+1}$ approximate $\alpha$ with error given by
\begin{equation}
\label{EquationBundedErrorClassicalContinuedFraction}
\frac{1}{2+a_{n+1}}\leq q_n^2\cdot|\alpha-p_n/q_n|\leq \frac{1}{a_{n+1}}
\text{ for any }
n\in\NN.
\end{equation}

\emph{Rosen continued fractions} where introduced in~\cite{Rosen}, in relation to diophantine approximation for \emph{Hecke groups}, proving in particular an extension of Equation~\eqref{EquationBundedErrorClassicalContinuedFraction}, which was later improved by \cite{Nakada}. Equation~\eqref{Equation(1/2)ApproximationsAreConvergents} was extended to Rosen continued fraction in~\cite{Lehner}, where the sharp constant replacing $1/2$ was obtained in~\cite{RosenSchmidt}. In this note we consider diophantine approximation for a general non-uniform lattice Fuchsian group, in relation to the so-called \emph{Bowen-Series expansion} of real numbers (\cite{BowenSeries}). We prove extensions of Equations~\eqref{Equation(1/2)ApproximationsAreConvergents} and~\eqref{EquationBundedErrorClassicalContinuedFraction}. Our estimates (Theorem~\ref{TheoremGoodApproximationsAndBoundaryExpansion} in \S~\ref{SectionGoodApproximations}) are used in~\cite{MarcheseDimension} to approximate the dimension of sets of \emph{badly approximable points} by the dimension of dynamically defined regular Cantor sets. The study of the high part of \emph{Markov and Lagrange spectra} is also a natural application, in the spirit of \cite{SheingornSchmidt}, \cite{ArtigianiMarcheseUlcigrai(Veech)} and \cite{ArtigianiMarcheseUlcigrai(Fuchsian)}. In general, Equations~\eqref{Equation(1/2)ApproximationsAreConvergents} and~\eqref{EquationBundedErrorClassicalContinuedFraction} translate diophantine properties into ergodic properties of the Bowen-Series expansion. 

\smallskip

Let $\sltwoc$ be the group of matrices 
\begin{equation}
\label{EquationCoefficientsSL(2,C)}
G=
\begin{pmatrix}
a & b \\
c & d  
\end{pmatrix}
\end{equation}
with $a,b,c,d\in\CC$ and $ad-bc=1$, where any such $G$ acts on points $z\in\CC\cup\{\infty\}$ by 
\begin{equation}
\label{EquationActionSL(2,C)}
G\cdot z:=\frac{az+b}{cz+d}.
\end{equation}
Denote $a=a(G)$, $b=b(G)$, $c=c(G)$ and $d=d(G)$ the coefficients of $G$ as in  Equation~\eqref{EquationCoefficientsSL(2,C)}. The group $\sltwor$ of $G$ with coefficients 
$a,b,c,d$ in $\RR$ acts by isometries on $\HH$ via Equation~\eqref{EquationActionSL(2,C)}, and inherits  a topology from the identification with the set of $(a,b,c,d)\in\RR^4$ with $ad-bc=1$. 
A \emph{Fuchsian group} is a discrete subgroup $\Gamma<\sltwor$. Referring to~\cite{KatokFuchsian}, 
we say that $\Gamma$ is a \emph{lattice} if it has a \emph{Dirichlet region} 
$\Omega\subset\HH$ with finite hyperbolic area. If $\Omega$ is not compact, then the lattice $\Gamma$ is said \emph{non-uniform}. In this case the intersection 
$
\overline{\Omega}\cap\partial\HH
$ 
is a finite non-empty set, whose elements are called the vertices \emph{at infinity} of $\Omega$. A point $z\in\RR\cup\{\infty\}$ is a parabolic fixed point for $\Gamma$ if there exists $P\in\Gamma$ parabolic with $P(z)=z$. Let $\cP_\Gamma$ be the set of parabolic fixed points of $\Gamma$, which is equal to the orbit under $\Gamma$ of the vertices at infinity of 
$\Omega$. The set $\cP_\Gamma$ is dense in $\RR$. Two points $z_1$ and $z_2$ in $\cP_\Gamma$ are \emph{equivalent} if $z_2=G(z_1)$ for some $G\in\Gamma$. Any non-uniform lattice $\Gamma$ has a finite number $p\geq1$ of equivalence classes $[z_1],\dots,[z_p]$ of parabolic fixed points, called the  \emph{cusps} of $\Gamma$.

\smallskip

Let $\Gamma$ be a non-uniform lattice with $p\geq1$ cusps. Fix a list  
$\cS=(A_1,\dots,A_p)$ of elements $A_k\in\sltwor$ such that the points
\begin{equation}
\label{EquationRepresentativesOfCusps}
z_k=A_k\cdot\infty
\quad
\textrm{ for }
\quad
k=1,\dots,p
\end{equation}
form a complete set 
$
\{z_1,\dots,z_p\}\subset\cP_\Gamma
$ 
of inequivalent parabolic fixed points. Any element of $\cP_\Gamma$ has the form $G\cdot z_k$ for some $G\in\Gamma$ and $k=1,\dots,p$. We have horoballs 
\footnote{A natural choice for $z_1,\dots,z_p$ is a maximal set of non-equivalent vertices of a fundamental domain which are not interior points of $\HH$. We can have $z_k=\infty$ and $B_k=H_\infty$, that is $A_k=\id$.}
$$
B_k:=A_k\big(\{z\in\HH:\im(z)>1\}\big)
\quad\text{ with }\quad
k=1,\dots,p,
$$
each $B_k$ being tangent to $\RR\cup\{\infty\}$ at $z_k$. Thus $G(B_k)$ is a ball tangent to the real line at $G\cdot z_k$ for any $G\in\Gamma$ with $G\cdot z_k\not=\infty$. These balls generalize Ford circles and we measure how their diameter shrinks to zero as $G$ varies in $\Gamma$ with the \emph{denominator}
$$
D(G\cdot z_k):=
\left\{
\begin{array}{ccc}
1/\sqrt{\diameter\big(G(B_k)\big)} & \text{if} & G\cdot z_k\not=\infty
\\
0  & \text{if} & G\cdot z_k=\infty.
\end{array}
\right.
$$ 
Recall that for any $T>0$ and any $G\in\sltwor$ with $c(G)\not=0$ we have 
\begin{equation}
\label{EquationDiameterHoroball}
\diameter\bigg(G\big(\big\{z\in\HH:\im(z)>T\}\big)\bigg)=\frac{1}{Tc^2(G)},
\end{equation}
where we refer to the notation of Equation~\eqref{EquationCoefficientsSL(2,C)}. Hence
\begin{equation}
\label{EquationDefinitionDenominator}
D(G\cdot z_k)=|c(GA_k)|
\quad\text{ for any }\quad
G\cdot z_k\in\cP_\Gamma.
\end{equation}
In~\cite{Patterson}, Patterson proves that there exists a constant $M=M(\Gamma,\cS)>0$ such that for any $Q>0$ big enough and any $\alpha\in\RR$ there exists $G\in\Gamma$ and $k\in\{1,\dots,p\}$ with 
$$
|\alpha-G\cdot z_k|
\leq
\frac{M}{D(G\cdot z_k)Q}
\quad
\textrm{ and }
\quad
0<D(G\cdot z_k)\leq Q.
$$ 
For $\Gamma=\sltwoz$, $\cS=\{\id\}$ and $M=1$ Patterson's Theorem gives the Classical Dirichlet Theorem. In general, for any $\alpha\in\RR$ we obtain infinitely many 
$G\cdot z_k\in\cP_\Gamma$ with 
\begin{equation}
\label{EquationAsymptoticDirichletPatterson}
|\alpha-G\cdot z_k|\leq \frac{M}{D^2(G\cdot z_k)}.
\end{equation}

The \emph{Bowen-Series expansion} (\cite{BowenSeries}) provides a coding $\alpha=[W_1,W_2,\dots]$ of a  real number $\alpha$, where for $r\geq1$ we call \emph{cuspidal words} the symbols $W_r$, which belong to a countable alphabet $\cW$ (definitions are in \S~\ref{SectionBoundaryExpansion} and \S~\ref{SectionGoodApproximations}). 
Cuspidal words $W\in\cW$, that where introduced in~\cite{ArtigianiMarcheseUlcigrai(Veech)} and \cite{ArtigianiMarcheseUlcigrai(Fuchsian)}, label a subset of elements $\{G_W:W\in\cW\}$ of $\Gamma$, which generalize the role played in the theory of classical continued fractions by the matrices 
$$
\begin{pmatrix}
1 & a_{2k+1} \\
0 & 1
\end{pmatrix}
\quad\text{ and }\quad
\begin{pmatrix}
1 & 0 \\
a_{2k} & 1
\end{pmatrix}
\quad\text{ with }\quad
a_{2k},a_{2k+1}\in\NN^\ast
\text{ for any }
k\in\NN.
$$
The coding is a continuous bijection $\Sigma\to\RR$, where $\Sigma\subset\cW^\NN$ is a subshift with \emph{aperiodic transition matrix} (see \cite{MarcheseDimension}). 
For $r\geq1$ the first $r$ symbols in the expansion of $\alpha=[W_1,W_2,\dots]$ define  
$
\zeta_r=\zeta_r(W_1,\dots,W_r)\in\cP_\Gamma
$, 
see Equation~\eqref{EquationConvergentBowenSeries}. This extends the classical notion of convergents $p_n/q_n$ of $\alpha$. 
The main result of this note is Theorem~\ref{TheoremGoodApproximationsAndBoundaryExpansion} 
in \S~\ref{SectionGoodApproximations}. We give the following preliminary statement (see also Remark~\ref{RemarkCovarianceIndependence}).

\begin{thmnonnum}[Theorem~\ref{TheoremGoodApproximationsAndBoundaryExpansion}]
Fix $\alpha=[W_1,W_2,\dots]$ which is not an element of $\cP_\Gamma$. The convergents 
$
\zeta_r=\zeta_r(W_1,\dots,W_r)
$ 
approximate $\alpha$ with error given by an analogous of 
Equation~\eqref{EquationBundedErrorClassicalContinuedFraction}. Moreover there exists a constant 
$\epsilon_0>0$ such that any $G\cdot z_k\in\cP_\Gamma$ satisfying 
Equation~\eqref{EquationAsymptoticDirichletPatterson} with $M=\epsilon_0$ belongs to the sequence 
$(\zeta_r)_{r\geq1}$.
\end{thmnonnum}

\subsection*{Acknowledgements}

The author is grateful to M. Artigiani and C. Ulcigrai.

\section{The Bowen-Series expansion}
\label{SectionBoundaryExpansion}

We follow \S~3 in~\cite{MarcheseDimension}, which is itself based on 
\S~2.4 in \cite{ArtigianiMarcheseUlcigrai(Veech)} and \S~2 in \cite{ArtigianiMarcheseUlcigrai(Fuchsian)}. The original construction is in \cite{BowenSeries}. 
Consider the unit disc $\DD:=\{z\in\CC:|z|<1\}$ and the map 
\begin{equation}
\label{EquationConjugationUpperHalfPlaneDisc}
\varphi:\HH\to\DD
\quad\text{;}\quad
\varphi(z):=\frac{z-i}{z+i}.
\end{equation}
The conjugated of $\sltwor$ under $\varphi$ is the group $\sugroup(1,1)$ of $F\in\glgroup(2,\CC)$ with 
\begin{equation}
\label{EquationCoefficientsSU(1,1)}
F=
\begin{pmatrix}
\alpha & \overline{\beta} \\
\beta & \overline{\alpha}  
\end{pmatrix}
\quad
\textrm{ with }
\quad
|\alpha|^2-|\beta|^2=1.
\end{equation}
Denote $\alpha=\alpha(F)$ and $\beta=\beta(F)$ the coefficients of $F$ as in Equation~\eqref{EquationCoefficientsSU(1,1)}.

\subsection{Isometric circles}
\label{SectionIsometricCircles}

Consider $F\in\sugroup(1,1)$ and $\alpha=\alpha(F)$, $\beta=\beta(F)$ as in 
Equation~\eqref{EquationCoefficientsSU(1,1)}. Assume $\beta\not=0$ and let 
$
\omega_F:=-\overline{\alpha}/\beta
$ 
be the pole of $F$. The \emph{isometric circle} $I_F$ of $F$ is the euclidean circle centered at 
$\omega_F$ with radius $\rho(F):=|\beta|^{-1}$, that is 
$$
I_F:=\{\xi\in\CC:|\xi-\omega_F|\}.
$$
According to Theorem~3.3.2 in \cite{KatokFuchsian} we have
$$
F(I_F)=I_{F^{-1}}
\quad
\textrm{ where }
\quad
\rho(F)=\rho(F^{-1})
\quad
\textrm{ and }
\quad
|\omega_{F^{-1}}|=|\omega_F|.
$$
Moreover $I_F\cap\DD$ is a geodesic of $\DD$ for any $F\in\sugroup(1,1)$, by Theorem~3.3.3 in \cite{KatokFuchsian}. Denote $U_F$ the disc in $\CC$ with 
$\partial U_F=I_F$, that is the interior of $I_F$. 

\subsection{Labelled ideal polygon}
\label{SectionLabelledIdealPolygon}

Let $\Gamma\subset\sugroup(1,1)$ be a non-uniform lattice. 
According to~\cite{Tukia}, there exist a free subgroup $\Gamma_0<\Gamma$ with finite index 
$[\Gamma_0:\Gamma]<+\infty$. See also \S~2.2 of~\cite{MarcheseDimension}. In particular $\beta(F)\not=0$ for any $F\in\Gamma_0$, referring to 
Equation~\eqref{EquationCoefficientsSU(1,1)}, so that the isometric circle $I_F$ and the disc $U_F$ introduced in \S~\ref{SectionIsometricCircles} are defined. The origin $0\in\DD$ is not a fixed point of any $F\in\Gamma_0$ and Theorem 3.3.5 in \cite{KatokFuchsian} implies that the set
\begin{equation}
\label{EquationDefinitionFundamentalPolygon}
\Omega_0:=\overline{\DD\setminus\bigcup_{F\in\Gamma_0}U_F}
\end{equation}
is a Dirichlet region for $\Gamma_0$. As it is explained in \S~2.1 and \S~2.4 of~\cite{MarcheseDimension}, the polygon $\Omega_0$ has finitely many sides, which we denote by the letter $s$. Moreover there is an even number $2d$ of sides. 
Thus $\Omega_0$ has $2d$ vertices, which we denote by the letter 
$\xi$. All vertices of $\Omega_0$ belong to $\partial\DD$, because $\Gamma_0$ is free. Any side $s$ is a complete geodesic in $\DD$ and for any such $s$ there exists an unique $F\in\Gamma$ such that $F(s)$ is an other side of $\Omega_0$ with $F(s)\not=s$. The sides $s$ and $F(s)$ are thus \emph{paired}. 
See Figure~\ref{FigureLabelledIdealPolygon}. The set of pairings generates $\Gamma_0$, according to Theorem~3.5.4 in \cite{KatokFuchsian}. In order to establish a convenient labelling, consider two finite alphabets $\cA_0$ and $\widehat{\cA_0}$, both with $d$ elements and a map 
$$
\iota:\cA_0\cup\widehat{\cA_0}\to\cA_0\cup\widehat{\cA_0}
\quad
\textrm{ with }
\quad
\iota^2=\id
\quad
\textrm{ and }
\quad
\iota(\cA_0)=\widehat{\cA_0},
$$
that is an involution of $\cA_0\cup\widehat{\cA_0}$ which exchanges $\cA_0$ with $\widehat{\cA_0}$. For convenience of notation, set $\cA:=\cA_0\cup\widehat{\cA_0}$ and for any $a\in\cA$, denote 
$\widehat{a}:=\iota(a)$.

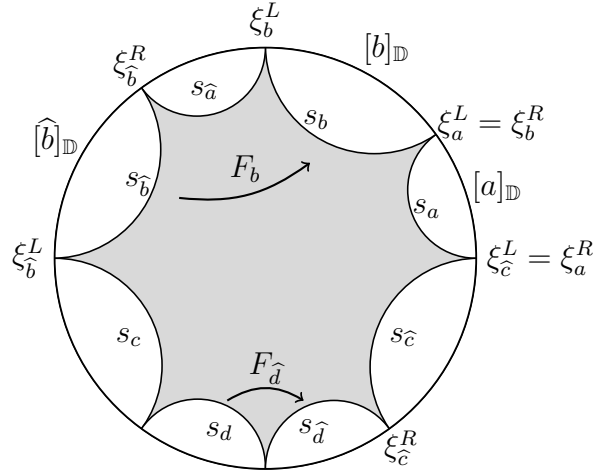
\begin{figure}
\begin{center}
\begin{tikzpicture}[scale=2.8]

\begin{scope}

\tkzDefPoint(0,0){O}
\tkzDefPoint(1,0){A}
\tkzDrawCircle[thick,fill=black!15](O,A)
\tkzClipCircle(O,A)



\tkzDefPoint({cos(0)},{sin(0)}){z0}
\tkzDefPoint({cos(0.2*pi)},{sin(0.2*pi)}){z1}
\tkzDefPoint({cos(0.5*pi)},{sin(0.5*pi)}){z2}
\tkzDefPoint({cos(0.7*pi)},{sin(0.7*pi)}){z3}
\tkzDefPoint({cos(1.0*pi)},{sin(1.0*pi)}){z4}
\tkzDefPoint({cos(1.3*pi)},{sin(1.3*pi)}){z5}
\tkzDefPoint({cos(1.5*pi)},{sin(1.5*pi)}){z6}
\tkzDefPoint({cos(1.7*pi)},{sin(1.7*pi)}){z7}

\tkzDrawCircle[orthogonal through=z0 and z1,fill=white](O,A)
\node[] at ({0.8*cos(17)},{0.8*sin(17)}) {$s_a$};
\tkzDrawCircle[orthogonal through=z1 and z2,fill=white](O,A)
\node[] at ({0.7*cos(70)},{0.7*sin(70)}) {$s_b$};
\tkzDrawCircle[orthogonal through=z2 and z3,fill=white](O,A)
\node[] at ({0.85*cos(110)},{0.85*sin(110)}) {$s_{\widehat{a}}$};
\tkzDrawCircle[orthogonal through=z3 and z4,fill=white](O,A)
\node[] at ({0.7*cos(150)},{0.7*sin(150)}) {$s_{\widehat{b}}$};
\tkzDrawCircle[orthogonal through=z4 and z5,fill=white](O,A)
\node[] at ({0.75*cos(210)},{0.75*sin(210)}) {$s_c$};
\tkzDrawCircle[orthogonal through=z5 and z6,fill=white](O,A)
\node[] at ({0.85*cos(255)},{0.85*sin(255)}) {$s_d$};
\tkzDrawCircle[orthogonal through=z6 and z7,fill=white](O,A)
\node[] at ({0.85*cos(285)},{0.85*sin(285)}) {$s_{\widehat{d}}$};
\tkzDrawCircle[orthogonal through=z7 and z0,fill=white](O,A)
\node[] at ({0.75*cos(330)},{0.75*sin(330)}) {$s_{\widehat{c}}$};

\draw[->,thick] ({0.7*cos(255)},{0.7*sin(255)}) .. controls 
({0.6*cos(265)},{0.6*sin(265)}) and ({0.6*cos(275)},{0.6*sin(275)}) 
.. ({0.7*cos(285)},{0.7*sin(285)}) {};

\node[] at ({0.65*cos(270)},{0.65*sin(270)}) [above] {$F_{\widehat{d}}$};

\draw[->,thick] ({0.5*cos(145)},{0.5*sin(145)}) .. controls 
({0.3*cos(120)},{0.3*sin(120)}) and ({0.3*cos(90)},{0.3*sin(90)}) 
.. ({0.5*cos(65)},{0.5*sin(65)}) {};

\node[] at ({0.55*cos(100)},{0.55*sin(100)}) [right,below] {$F_{b}$};

\end{scope}

\begin{scope}

\tkzDefPoint(0,0){O}
\tkzDefPoint(1,0){A}
\tkzDrawCircle(O,A)

\node[] at ({cos(0)},{sin(0)}) [right] {$\xi^L_{\widehat{c}}=\xi^R_a$};
\node[] at ({cos(40)},{sin(40)}) [right] {$\xi^L_a=\xi^R_b$};
\node[] at ({cos(90)},{sin(90)}) [above] {$\xi^L_b$};
\node[] at ({cos(130)},{sin(130)}) [above] {$\xi^R_{\widehat{b}}$};
\node[] at ({cos(180)},{sin(180)}) [left] {$\xi^L_{\widehat{b}}$};
\node[] at ({cos(310)},{sin(310)}) [right,below] {$\xi^R_{\widehat{c}}$};

\node[] at ({1.15*cos(17)},{1.15*sin(17)}) {$[a]_\DD$};
\node[] at ({1.15*cos(60)},{1.15*sin(60)}) {$[b]_\DD$};
\node[] at ({1.15*cos(150)},{1.15*sin(150)}) {$[\widehat{b}]_\DD$};

\end{scope}
\end{tikzpicture}

\end{center}
\caption{Ideal polygon labelled by 
$
\cA=\{a,b,c,d,\widehat{a},\widehat{b},\widehat{c},\widehat{d}\}
$.
}
\label{FigureLabelledIdealPolygon}
\end{figure}

\begin{itemize}
\item
Label the sides of $\Omega_0$ by the letters in $\cA$, so that for any 
$
a\in\cA
$ 
the sides $s_a$ and $s_{\widehat{a}}$ are those which are paired by the action of $\Gamma_0$.  
\item
For any pair of sides $s_a$ and $s_{\widehat{a}}$ as above, let $F_a$ be the unique element of 
$\Gamma_0$ such that
\begin{equation}
\label{EquationPairingPolygonSides}
F_a(s_{\widehat{a}})=s_a.
\end{equation}
\item
For any $a\in\cA$ we have $F_{\widehat{a}}=F_a^{-1}$, and the latter form a set of generators for $\Gamma_0$. 
\end{itemize}

In the following we denote  
$
\Omega_\DD:=\Omega_0\subset\DD
$ 
the labelled ideal polygon defined above and 
$
\Omega_\HH:=\varphi^{-1}(\Omega_\DD)\subset\HH
$ 
its preimage under the map in 
Equation~\eqref{EquationConjugationUpperHalfPlaneDisc}.

\subsection{The boundary map}
\label{SectionTheBoundaryMap}

Parametrize arcs $J\subset\partial\DD$ by $t\mapsto e^{-it}$ with $t\in(x,y)$. Set $\inf J:=e^{-ix}$ and $\sup J:=e^{-iy}$. We say that $J$ is \emph{right open} if $\inf J\in J$ and $\sup J\not\in J$. 
Let $\Gamma_0<\Gamma$ be a finite index free subgroup and $\Omega_\DD$ be an ideal polygon for 
$\Gamma_0$ labelled by $\cA$, as in \S~\ref{SectionLabelledIdealPolygon}. 

\smallskip

For $a\in\cA$ let $F_a$ be the map in Equation~\eqref{EquationPairingPolygonSides}. 
Let $I_{F_a}$ be the isometric circle of $F_a$ and $U_{F_a}$ be its interior, as in \S~\ref{SectionIsometricCircles}. 
Recall that  
$
s_{\widehat{a}}=I_{F_a}\cap\DD
$ 
and 
$
s_{a}=I_{F_{\widehat{a}}}\cap\DD
$. 
Let $[a]_\DD$ be the right open arc of $\partial\DD$ cut by the side $s_a$, that is
$$
[a]_\DD:=U_{F_{\widehat{a}}}\cap\partial\DD.
$$
Set $\xi_a^L:=\inf [a]_\DD$ and $\xi_a^R:=\sup [a]_\DD$. 
Figure~\ref{FigureLabelledIdealPolygon} shows examples of such notation. In order to take account of the cyclic order in $\partial\DD$ of the arcs $[a]_\DD$, fix $a_0\in\cA$ and define a map 
$o:\cA\to\ZZ/2d\ZZ$ setting $o(a_0):=0$ and
\begin{equation}
\label{EquationCyclicOrderLetters}
o(b)=o(a)+1 \mod2d
\quad
\textrm{ for }
\quad
a,b\in\cA
\quad
\textrm{ with }
\quad
\xi^R_a=\xi^L_b.
\end{equation}
We have $F_a(I_{F_a})=I_{F_{\widehat{a}}}$ for any $a\in\cA$, 
thus $F_a$ sends the complement of $[\widehat{a}]_\DD$ to $[a]_\DD$, that is 
\begin{equation}
\label{EquationActionGeneratorsBoundary}
F_a\big(\partial\DD\setminus [\widehat{a}]_\DD\big)=[a]_\DD.
\end{equation}
The \emph{Bowen-Series map} is the map $\cB\cS:\partial\DD\to\partial\DD$ defined by
\begin{equation}
\label{EquationBowenSeriesMap}
\cB\cS(\xi):=F_a^{-1}(\xi)
\quad
\textrm{iff}
\quad
\xi\in [a]_\DD. 
\end{equation}
The \emph{boundary expansion} of  a point $\xi\in\partial\DD$ is the sequence $(a_k)_{k\in\NN}$ of letters $a_k\in\cA$ with 
\begin{equation}
\label{EquationDefinitionBoundaryExpansion}
\cB\cS^k(\xi)\in [a_k]_\DD
\quad
\textrm{ for any }
k\in \NN.
\end{equation}
By Equation~\eqref{EquationActionGeneratorsBoundary}, any such sequence satisfies the so-called \emph{no backtracking Condition}:
\begin{equation}
\label{EqNobacktrack}
a_{k+1}\not=\widehat{a_k}
\textrm{ for any }
k\in\NN.
\end{equation}
A finite word $(a_0,\dots,a_n)$ satisfying Condition~\eqref{EqNobacktrack} corresponds to a \emph{factor} of the map $\cB\cS:\partial\DD\to\partial\DD$, that is a finite concatenation  
$
F_{a_n}^{-1}\circ\dots\circ F_{a_0}^{-1}
$ 
arising from iterations of $\cB\cS$. We call \emph{admissible word}, or simply \emph{word}, any finite or infinite word in the letters of $\cA$ satisfying 
Condition~\eqref{EqNobacktrack}. We use the notation
$$
F_{a_0,\dots,a_n}:=F_{a_0}\circ\dots\circ F_{a_n}\in \Gamma_0.
$$
Define the right open arc 
$
[a_0,\dots,a_n]_\DD
$ 
as the set of $\xi\in\partial\DD$ such that $\cB\cS^k(\xi)\in[a_k]_\DD$ for any $k=0,\dots,n$, that is
\begin{equation}
\label{EquationDefinitionCylinderBowenSeries}
[a_0,\dots,a_n]_\DD
:=
F_{a_0,\dots,a_{n-1}}[a_n]_\DD
=
F_{a_0,\dots,a_n}\big(\partial\DD\setminus[\widehat{a_n}]_\DD\big).
\end{equation}
Two such arcs satisfy  
$
[a_0,\dots,a_n]_\DD\subset [b_0,\dots,b_m]_\DD
$ 
if and only if $m\geq n$ and $a_k=b_k$ for any $k=0,\dots,n$. It is easy to see that   
$[a_0,\dots,a_n]_\DD$ shrinks to a point as $n\to\infty$. See Lemma~3.1 in~\cite{MarcheseDimension} for a proof. A sequence $(a_k)_{k\in\NN}$ satisfying Condition~\eqref{EqNobacktrack} corresponds to a point 
$
\xi=[a_0,a_1,\dots]_\DD
$ 
in $\partial\DD$, where we use the notation   
$$
[a_0,a_1,\dots]_\DD:=\bigcap_{n\in\NN}[a_0\dots,a_n]_\DD.
$$
Conversely, if $(a_k)_{k\in\NN}$ is the boundary expansion of $\xi\in\partial\DD$, then 
$\xi=[a_0,a_1,\dots]_\DD$. The Bowen-Series map $\cB\cS$ is the shift on the space of admissible infinite words.

\subsection{Cuspidal words}
\label{SectionCuspidalWords}

Consider the map $o:\cA\to\ZZ/2d\ZZ$ in Equation~\eqref{EquationCyclicOrderLetters}. 
Lemma~\ref{LemmaCombinatorialPropertiesCuspidal} below is an easy consequence of the definitions in 
\S~\ref{SectionTheBoundaryMap}. See Lemma~3.2 in~\cite{MarcheseDimension} for a proof.

\begin{lemma}
\label{LemmaCombinatorialPropertiesCuspidal}
Let $(a_0,\dots,a_n)$ be a word satisfying Condition~\eqref{EqNobacktrack} with $n\geq1$ and $a_0=a_n$. The map $F_{a_0,\dots,a_{n-1}}$ is a parabolic element of $\Gamma_0$ fixing $\xi^R_{a_0}$ if and only if 
\begin{equation}
\label{EquationRightCuspidalWord(-)}
o(a_{k+1})=o(\widehat{a_k})-1
\quad
\textrm{ for any }
\quad
k=0,\dots,n-1.
\end{equation}
The map $F_{a_0,\dots,a_{n-1}}$ is a parabolic element of $\Gamma_0$ fixing $\xi^L_{a_0}$ if and only if 
\begin{equation}
\label{EquationLeftCuspidalWord(+)}
o(a_{k+1})=o(\widehat{a_k})+1
\quad
\textrm{ for any }
\quad
k=0,\dots,n-1.
\end{equation}
\end{lemma}

Let $W=(a_0,\dots,a_n)$ be an admissible word. We say that $W$ is a \emph{cuspidal word} if it is the initial factor of an admissible word $(a_0,\dots,a_m)$ with $m\geq n$ such that 
$F_{a_0,\dots,a_m}$ is a parabolic element of $\Gamma_0$ fixing a vertex of $\Omega_\DD$. 

\begin{itemize}
\item
If $n\geq1$ and Equation~\eqref{EquationRightCuspidalWord(-)} is satisfied, we say that $W$ is a \emph{right cuspidal word}. In this case we define its type by $\varepsilon(W):=R$ and we set 
$\xi_W:=\xi^R_{a_0}$.
\item
If $n\geq1$ and Equation~\eqref{EquationLeftCuspidalWord(+)} is satisfied, we say that $W$ is a \emph{left cuspidal word}. In this case we define its type by $\varepsilon(W):=L$ and we set 
$\xi_W:=\xi^L_{a_0}$.
\item
If $n=0$, that is $W=(a_0)$ has just one letter, the type $\varepsilon(W)$ is not defined. We set by convention $\xi_W:=\xi^R_{a_0}$.
\end{itemize}

If $W=(a_0,\dots,a_n)$ is cuspidal with $n\geq1$, Lemma~\ref{LemmaCombinatorialPropertiesCuspidal} implies 
$
\xi^{\varepsilon(W)}_{a_k}=F_{a_k}\cdot\xi^{\varepsilon(W)}_{a_{k+1}}
$ 
for any $k=0,\dots,n-1$ and it follows 
\begin{equation}
\label{EquationCommonEndpointsCuspidalWords}
\xi_W=\partial[a_0]_\DD\cap\partial[a_0,a_1]_\DD\cap\dots\cap\partial[a_0,\dots,a_n]_\DD,
\end{equation}
that is the $n+1$ arcs above share $\xi_W$ as common endpoint (see also \S~2.4 in \cite{ArtigianiMarcheseUlcigrai(Fuchsian)} and \S~4.3 in \cite{ArtigianiMarcheseUlcigrai(Veech)}). 
A sequence $(a_n)_{n \in \NN}$ is said \emph{cuspidal} if any initial factor $(a_0,\dots,a_n)$ with $n\in\NN$ is a cuspidal word, and \emph{eventually cuspidal} if there exists 
$k \in \NN$ such that $(a_{n+k})_{n \in \NN}$ is a cuspidal sequence.

\subsection{The cuspidal acceleration}
\label{SectionCuspidalAcceleration}

If $W=(b_0,\dots,b_m)$ and $W'=(a_0,\dots,a_n)$ are words with 
$a_0\not=\widehat{b_m}$, define the word 
$
W\ast W':=(b_0,\dots,b_m,a_0,\dots,a_n)
$. 
Let $(a_n)_{n\in\NN}$ be a sequence satisfying Condition \eqref{EqNobacktrack} and not eventually cuspidal. 

\begin{description}
\item[Initial step]
Set $n(0):=0$. Let $n(1)\in\NN$ be the maximal integer $n(1)\geq 1$ such that 
$
(a_0,\dots,a_{n(1)-1})
$ 
is cuspidal, then set
$$
W_0:=(a_0,\dots,a_{n(1)-1}).
$$
\item[Recursive step]
Fix $r\geq1$ and assume that the instants $n(0)<\dots<n(r)$ and the cuspidal words $W_0,\dots,W_{r-1}$ are defined. Define $n(r+1)\geq n(r)+1$ as the maximal integer such that 
$
[a_{n(r)},\dots,a_{n(r+1)-1}]
$ 
is cuspidal, then set 
$$
W_r:=(a_{n(r)},\dots,a_{n(r+1)-1}).
$$
\end{description}

The sequence of words $(W_r)_{r\in\NN}$ is called the \emph{cuspidal decomposition} of 
$(a_n)_{n\in\NN}$. We have of course 
$
a_0,a_1,a_2\dots=W_0\ast W_1\ast\dots
$. 
For any $\xi=[a_0,a_1,\dots]_\DD$, if $(W_r)_{r\in\NN}$ is the cuspidal decomposition of 
$(a_n)_{n\in\NN}$, we write
\begin{equation}
\label{EquationCuspidalCuttingSequence(UnitDisc)}
\xi=[a_0,a_1,\dots]_\DD=[W_0,W_1,\dots]_\DD.
\end{equation}

\begin{remark}
\label{RemarkConcatenationCuspidalWords}
If 
$
W_{r-1}:=(a_{n(r-1)},\dots,a_{n(r)-1})
$ 
and 
$
W_{r}:=(a_{n(r)},\dots,a_{n(r+1)-1})
$ 
are two consecutive cuspidal words in the cuspidal decomposition of a sequence $(a_n)_{n\in\NN}$ satisfying Condition \eqref{EqNobacktrack}, then the word 
$
(a_{n(r)-1},a_{n(r)},\dots,a_{n(r+1)-1})
$ 
can be cuspidal.
\end{remark}

\section{The main Theorem}
\label{SectionGoodApproximations}

The tools in \S~\ref{SectionBoundaryExpansion} induce a boundary expansion on $\RR$. In terms of such expansion we state our main Theorem~\ref{TheoremGoodApproximationsAndBoundaryExpansion}. 
Let $\Gamma_0<\Gamma$ be the free subgroup and $\Omega_\DD\subset\DD$ be the ideal polygon in \S~\ref{SectionLabelledIdealPolygon}. 
Recall that $\cP_{\Gamma_0}=\Gamma_0(\Omega_\DD\cap\partial\DD)$ by Theorem~4.2.5 in \cite{KatokFuchsian}. Since $\Gamma_0$ has finite index in $\Gamma$ then the two groups have the same set of parabolic fixed points, that is
\begin{equation}
\label{EquationSameParabolicFixedPoints}
\cP_\Gamma
=
\Gamma_0(\Omega_\DD\cap\partial\DD).
\end{equation}

\subsection{Geometric length of cuspidal words and main statement}

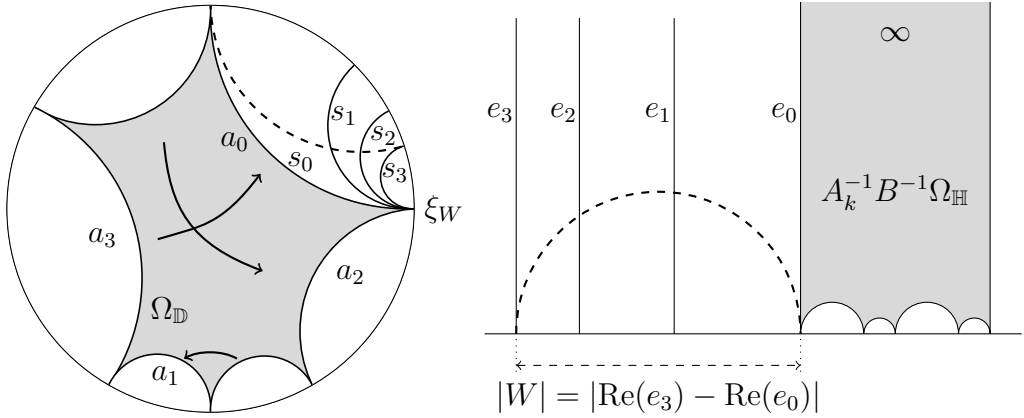
\begin{figure}
\begin{center}
\begin{tikzpicture}[scale=2.7]

\begin{scope}

\tkzDefPoint(0,0){O}
\tkzDefPoint(1,0){A}
\tkzDrawCircle(O,A)

\node[] at (1,0) [right]{$\xi_W$};

\end{scope}

\begin{scope}

\tkzDefPoint(0,0){O}
\tkzDefPoint(1,0){A}
\tkzDrawCircle[thick,fill=black!15](O,A)

\tkzClipCircle(O,A)


\node[] at (-0.2,-0.5) {$\Omega_\DD$};

\tkzDefPoint({cos(0)},{sin(0)}){z0}
\tkzDefPoint({cos(pi/2)},{sin(pi/2)}){z1}
\tkzDefPoint({cos(5*pi/6)},{sin(5*pi/6)}){z2}
\tkzDefPoint({cos(8*pi/6)},{sin(8*pi/6)}){z3}
\tkzDefPoint({cos(3*pi/2)},{sin(3*pi/2)}){z4}
\tkzDefPoint({cos(10*pi/6)},{sin(10*pi/6)}){z5}

\tkzDrawCircle[orthogonal through=z0 and z1,fill=white](O,A)
\node[] at ({0.5*cos(27.5)},{0.5*sin(27.5)}) {$s_0$};
\node[] at ({0.35*cos(70)},{0.35*sin(70)}) {$a_0$};
\tkzDrawCircle[orthogonal through=z1 and z2,fill=white](O,A)
\tkzDrawCircle[orthogonal through=z2 and z3,fill=white](O,A)
\node[] at ({0.55*cos(195)},{0.55*sin(195)}) {$a_3$};
\tkzDrawCircle[orthogonal through=z3 and z4,fill=white](O,A)
\node[] at ({0.85*cos(255)},{0.85*sin(255)}) {$a_1$};
\tkzDrawCircle[orthogonal through=z4 and z5,fill=white](O,A)
\tkzDrawCircle[orthogonal through=z5 and z0,fill=white](O,A)
\node[right] at ({0.65*cos(330)},{0.65*sin(330)}) {$a_2$};

\draw[->,thick] ({0.75*cos(280)},{0.75*sin(280)}) .. controls 
({0.7*cos(275)},{0.7*sin(275)}) and ({0.7*cos(265)},{0.7*sin(265)}) 
.. ({0.75*cos(260)},{0.75*sin(260)}) {};

\draw[->,thick] ({0.4*cos(125)},{0.4*sin(125)}) .. controls 
({0.2*cos(180)},{0.2*sin(180)}) and ({0.2*cos(240)},{0.2*sin(240)}) 
.. ({0.4*cos(310)},{0.4*sin(310)}) {};

\draw[->,thick] ({0.3*cos(210)},{0.3*sin(210)}) .. controls 
({0.1*cos(240)},{0.1*sin(240)}) and ({0.1*cos(300)},{0.1*sin(300)}) 
.. ({0.3*cos(35)},{0.3*sin(35)}) {};


\tkzDefPoint({cos(pi/4)},{sin(pi/4)}){z6}
\tkzDefPoint({cos(pi/8)},{sin(pi/6)}){z7}
\tkzDefPoint({cos(pi/10)},{sin(pi/10)}){z8}

\tkzDrawCircle[orthogonal through=z1 and z8,thick,dashed](O,A)
\tkzDrawCircle[orthogonal through=z0 and z6](O,A)
\node[] at ({0.8*cos(35)},{0.8*sin(35)}) {$s_1$};
\tkzDrawCircle[orthogonal through=z0 and z7](O,A)
\node[] at ({0.925*cos(23)},{0.925*sin(23)}) {$s_2$};
\tkzDrawCircle[orthogonal through=z0 and z8](O,A)
\node[] at ({0.925*cos(11)},{0.925*sin(11)}) {$s_3$};

\end{scope}

\end{tikzpicture}
\begin{tikzpicture}[scale=0.42]

\clip[shift={(0,2)}](-1,-4.5) rectangle (16,8.5);

\draw[thin] (-1,0) -- (16,0);

\draw (0,0) -- (0,10);
\node[] at (-0.5,7) {$e_3$};
\draw (2,0) -- (2,10);
\node[] at (1.5,7) {$e_2$};
\draw (5,0) -- (5,10);
\node[] at (4.5,7) {$e_1$};
\node[] at (8.5,7) {$e_0$};

\draw[<->,dashed] (0,-1) -- (9,-1);
\node[] at (4.5,-2) [] {$|W|=|\re(e_3)-\re(e_0)|$};
\draw[thin,dotted] (0,-1.2) -- (0,0);
\draw[thin,dotted] (9,-1.2) -- (9,0);

\filldraw[fill=black!15!white,draw=black] 
(15,11) -- (15,0)
arc (0:180:0.5) 
arc  (0:180:1)
arc (0:180:0.5) 
arc (0:180:1) -- (9,11) -- (15,11);

\node[] at (12,9.5) {$\infty$};

\draw[thick,dashed] (0,0) arc (180:0:4.5);

\node[] at (12,4.5) {$A_k^{-1}B^{-1}\Omega_\HH$};

\end{tikzpicture}

\end{center}
\caption{Geometric length $|W|$ of a right cuspidal word 
$
W=(a_0,a_1,a_2,a_3)
$. 
The arrows inside $\Omega_\DD$ represent the action of $F_{a_0},F_{a_1},F_{a_2}$. The arcs $s_0:=s_{a_0}$, 
$
s_1:=F_{a_0}(s_{a_1})
$,
$
s_2:=F_{a_0,a_1}(s_{a_2})
$ 
and 
$
s_3:=F_{a_0,a_1,a_2}(s_{a_3})
$ 
share the common vertex $\xi_W$, which is sent to 
$\infty$ under the map $A_k^{-1}B^{-1}\varphi^{-1}$. Thus the arcs $s_0,s_1,s_2,s_3$ in $\DD$ are sent to parallel vertical arcs $e_i:=\varphi^{-1}(s_i)$ in $\HH$.}
\label{FigureGeometricLengthCuspidalWord}
\end{figure}

Fix $\cS=(A_1,\dots,A_p)$ as in Equation~\eqref{EquationRepresentativesOfCusps}. 
Let 
$
\Omega_\HH:=\varphi^{-1}(\Omega_\DD)\subset\HH
$ 
be the pre-image of $\Omega_\DD$ under the map in  
Equation~\eqref{EquationConjugationUpperHalfPlaneDisc}. 
Any vertex $\xi$ of $\Omega_\DD$ corresponds to an unique vertex $\zeta=\varphi^{-1}(\xi)$ of 
$\Omega_\HH$. For any such vertex $\zeta$ consider $B\in\Gamma$ and $k\in\{1,\dots,p\}$ with 
\begin{equation}
\label{EquationVerticesDomainInfinity}
\zeta=BA_k\cdot \infty
\end{equation}
Any side $s_a$ of $\Omega_\DD$ corresponds to an unique side $e_{a}:=\varphi^{-1}(s_a)$ of $\Omega_\HH$, where $a\in\cA$. 
If $BA_k\cdot \infty=B'A_j\cdot \infty$, then $j=k$. Moreover \footnote{In any Fuchsian group $\Gamma$ with cusps, if $G\in\Gamma$ satisfies $G\cdot\zeta=\zeta$ for some $\zeta\in\cP_\Gamma$, then 
$G$ is parabolic.} we have $B'=BP$, where $P\in\Gamma$ is parabolic fixing $A_k\cdot\infty$. Hence the map $z\mapsto A_k^{-1}PA_k(z)$ is an horizontal translation in $\HH$. 
If $s$ and $s'$ are geodesics in $\DD$ having $\xi$ as common endpoint, then their pre-images in $\HH$ under 
$\varphi\circ B\circ A_k$ are parallel vertical half lines whose distance does not depend on the choice of $B$ in Equation~\eqref{EquationVerticesDomainInfinity}. We have a well defined positive real number 
$$
\Delta(s,s',\xi):=
\left|
\re\big(A_k^{-1}B^{-1}\varphi^{-1}(s)\big)
-
\re\big(A_k^{-1}B^{-1}\varphi^{-1}(s')\big)
\right|.
$$

Fix a cuspidal word  $W=(a_0,\dots,a_{n})$ and the vertex $\xi_W$ of $\Omega_\DD$ associated to $W$ in \S~\ref{SectionCuspidalWords}. For $n\geq1$ 
Equation~\eqref{EquationCommonEndpointsCuspidalWords} implies that the geodesics 
$
s_{a_0},F_{a_0}(s_{a_1}),\dots,F_{a_0,\dots,a_{n-1}}(s_{a_n})
$ 
all have $\xi_W$ as common endpoint. See Figure~\ref{FigureGeometricLengthCuspidalWord}. Define the \emph{geometric length} $|W|\geq0$ of $W$ as
\begin{equation}
\label{EquationDefinitionGeometricLenght}
|W|:=
\left\{
\begin{array}{ll}
\Delta\big(s_{a_0},F_{a_0,\dots,a_{n-1}}(s_{a_n}),\xi_W\big)
&\text{ if }n\geq1
\\
0
&\text{ if }n=0.
\end{array}
\right.
\end{equation}

For $a\in\cA$ set 
$
G_a=\varphi^{-1}\circ F_a\circ\varphi
$. 
Set 
$
G_{a_0,\dots,a_n}:=G_{a_0}\circ\dots\circ G_{a_n}
$ 
for any word $(a_0,\dots,a_n)$ and 
$
G_{W_0,\dots,W_r}=G_{a_0,\dots,a_n}
$ 
if $(a_0,\dots,a_n)=W_0\ast\dots\ast W_r$. Define the interval
$$
[a_0,\dots,a_n]_\HH:=
\varphi^{-1}\big([a_0,\dots,a_n]_\DD\big)=
G_{a_0,\dots,a_n}\big(\partial\HH\setminus[\widehat{a_n}]_\HH\big).
$$

Set  
$
[a_0,a_1,\dots]_\HH:=\varphi^{-1}\big([a_0,a_1,\dots]_\DD\big)
$, 
that is encode $\alpha\in\RR$ by the same cutting sequence as $\varphi(\alpha)\in\DD$. 
If $(a_n)_{n\in\NN}$ has cuspidal decomposition $(W_r)_{r\in\NN}$, 
Equation~\eqref{EquationCuspidalCuttingSequence(UnitDisc)} becomes 
\begin{equation}
\label{EquationCuspidalCuttingSequence(UpperHalfPlane)}
\alpha=[W_0,W_1,\dots]_\HH:=[a_0,a_1,\dots]_\HH.
\end{equation}
For $r\in\NN$ let $W_r$ be the $r$-th cuspidal word. Set 
$
\zeta_{W_r}:=\varphi^{-1}(\xi_{W_r})
$. 
The convergents of $\alpha$ are  
\begin{equation}
\label{EquationConvergentBowenSeries}
\zeta_r:=G_{W_0,\dots,W_{r-1}}\cdot\zeta_{W_r}
\quad\text{ ; }\quad
r\in\NN.
\end{equation}

For $k=1,\dots,p$ let $\mu_k>0$ be such that the primitive parabolic element 
$P_k\in A_k \Gamma A_k^{-1}$ fixing $\infty$ acts by $P_k(z)=z+\mu_k$. Set 
$\mu:=\max\{\mu_1,\dots,\mu_p\}$.

\begin{theorem}
\label{TheoremGoodApproximationsAndBoundaryExpansion}
For any $r\in\NN$ with $|W_r|>0$ we have 
\begin{equation}
\label{EquationGoodApproximationsAndBoundaryExpansion(1)}
\frac{1}{|W_r|+2\mu}
\leq
D(G_{W_0,\dots,W_{r-1}}\cdot\zeta_{W_r})^2
\cdot
|\alpha-G_{W_0,\dots,W_{r-1}}\cdot \zeta_{W_r}|
\leq 
\frac{1}{|W_r|}.
\end{equation}
Moreover there exists $\epsilon_0>0$ depending only on $\Omega_\DD$ and on $\cS$, such that for any $G\in\Gamma$ and $k=1,\dots,p$ with $D(G\cdot z_k)\not=0$ the condition 
$$
D(G\cdot z_k)^2\cdot|\alpha-G\cdot z_k|<\epsilon_0
$$
implies that there exists some $r\in\NN$ such that  
\begin{equation}
\label{EquationGoodApproximationsAndBoundaryExpansion(2)}
G\cdot z_k=G_{W_0,\dots,W_{r-1}}\cdot\zeta_{W_r}
\quad
\textrm{ where }
\quad
|W_r|>0.
\end{equation}
\end{theorem}

\begin{remark}
\label{RemarkCovarianceIndependence}
Equation~\eqref{EquationGoodApproximationsAndBoundaryExpansion(1)} shows that geometric length and denominators satisfy a form of covariance under the choice of the set $\cS$ in Equation~\eqref{EquationRepresentativesOfCusps}. 
Equation~\eqref{EquationGoodApproximationsAndBoundaryExpansion(2)} shows that good approximations $\zeta_r$ don't depend on the choice of the subgroup $\Gamma_0$. 
\end{remark}

\subsection{Reduced form of parabolic fixed points}
\label{SectionParabolicFixedPointsAndFiniteIndexFreeSubgroup}

Fix $G\cdot z_k\in\cP_\Gamma$. Recall Equation~\eqref{EquationSameParabolicFixedPoints} and write 
elements of $\Gamma_0$ in the generators $\{G_a:a\in\cA\}$. There exists an unique admissible word $b_0,\dots,b_m$ and a vertex $\zeta$ of 
$\Omega_\HH$ which is not an endpoint of $e_{\widehat{b_m}}$ such that 
$$
G\cdot z_k=G_{b_0,\dots,b_m}\cdot\zeta.
$$
The representation above is called the \emph{reduced form} of the parabolic fixed point $G\cdot z_k$. 
In the next Lemmas~\ref{LemmaDistancePoleReducedForm} and~\ref{LemmaInequalityDenominatorsReducedForm}, let $(b_0,\dots,b_m)$ be a non-trivial admissible word and let $\zeta_0$ be a vertex of $\Omega_\HH$ which is not an endpoint of $e_{\widehat{b_m}}$, so that 
$
G_{b_0,\dots,b_m}\cdot\zeta_0
$ 
is a parabolic fixed point written in its reduced form and different from $\infty$.

\begin{lemma}
\label{LemmaDistancePoleReducedForm}
There exists a constant $\kappa_1>0$, depending only on $\Omega_\HH$, such that 
$$
\left|
\zeta_0-G_{b_0,\dots,b_m}^{-1}\cdot\infty
\right|
\geq \kappa_1,
$$
that is the vertex $\zeta_0$ and the pole of $G_{b_0,\dots,b_m}$ stay at distance uniformly bounded from below.
\end{lemma}

\begin{proof}
We have 
$
G_{b_0,\dots,b_m}\big(\RR\setminus[\widehat{b_m}]_\HH\big)=[b_0,\dots,b_m]_\HH
$ 
By Equation~\eqref{EquationDefinitionCylinderBowenSeries}. 
Since $\infty$ does not belong to the interior of $[b_0,\dots,b_m]_\HH$ then the pole of $G_{b_0,\dots,b_m}$ belongs to the closure of 
$[\widehat{b_m}]_\HH$. The Lemma follows because $\zeta_0$ is a vertex of $\Omega_\HH$ different from the endpoints of $e_{\widehat{b_m}}$.
\end{proof}

\begin{lemma}
\label{LemmaInequalityDenominatorsReducedForm}
There exists a constant $\kappa_2>0$, depending only on $\Omega_\HH$ and on $\cS$, such that the following holds. 
\begin{enumerate}
\item
If $\zeta_1$ is a vertex of $\Omega_\HH$ different from $\zeta_0$, then  
$$
D(G_{b_0,\dots,b_m}\cdot\zeta_0)
\geq 
\kappa_2\cdot D(G_{b_0,\dots,b_m}\cdot\zeta_1).
$$
\item
If $b_{m+1}$ satisfies $b_{m+1}\not=\widehat{b_m}$ and $\zeta_2$ is a vertex of $\Omega_\HH$ with 
$
G_{b_{m+1}}\cdot\zeta_2\not=\zeta_0
$,  
then   
$$
D(G_{b_0,\dots,b_m}\cdot\zeta_0)
\geq 
\kappa_2\cdot D(G_{b_0,\dots,b_m,b_{m+1}}\cdot\zeta_2).
$$
\end{enumerate}
\end{lemma}

\begin{proof}
We prove Part (1). Set $G:=G_{b_0,\dots,b_m}$, 
$
\zeta:=G\cdot\zeta_0
$, 
and 
$
\zeta':=G\cdot\zeta_1
$. 
If $\zeta'=\infty$ then the statement is trivially true. If $D(G\cdot\zeta_1)\not=0$, let 
$\zeta_0=B_0A_k\cdot\infty$ and $\zeta_1=B_1A_j\cdot\infty$ as in 
Equation~\eqref{EquationVerticesDomainInfinity}. Referring to  
Equation~\eqref{EquationCoefficientsSL(2,C)}, let $c,d$ be the entries of $G$. Let $a_0,c_0$ and $a_1,c_1$ be the entries of $B_0A_k$ and $B_1A_j$ respectively. We prove an upper bound for
$$
\frac{D(G_{b_0,\dots,b_m}\cdot\zeta_1)}{D(G_{b_0,\dots,b_m}\cdot\zeta_0)}
=
\left|
\frac{ca_1+dc_1}{ca_0+dc_0}
\right|.
$$
We cannot have $c_0=c_1=0$, because $\zeta_0\not=\zeta_1$ and in particular $\zeta_0$, $\zeta_1$ cannot be both equal to $\infty$. Moreover $G\cdot\zeta_0$, $G\cdot\zeta_1$ are both different from $\infty$, thus condition $c=0$ implies $c_0,c_1\not=0$. Hence for $c=0$ Part (1) follows because the ratio above equals $|c_1/c_0|$, which varies in a finite set of values and is therefore bounded from above. If $c,c_0,c_1\not=0$ then 
$$
\left|
\frac{ca_1+dc_1}{ca_0+dc_0}
\right|
=
\left|
\frac{c_1}{c_0}
\right|
\cdot
\left|
\frac{(a_1/c_1)-(-d/c)}{(a_0/c_0)-(-d/c)}
\right|
=
\left|
\frac{c_1}{c_0}
\right|
\cdot
\left|
\frac{\zeta_1-(G^{-1}\cdot\infty)}{\zeta_0-(G^{-1}\cdot\infty)}
\right|.
$$
In this case Part (1) follows because $|c_1/c_0|$ is bounded from above, and 
Lemma~\ref{LemmaDistancePoleReducedForm} gives a lower bound for the denominator of the second factor (the numerator is not bounded, but as it increases the ratio converges to $1$). If $c,c_0\not=0$ and $c_1=0$ then Lemma~\ref{LemmaDistancePoleReducedForm} gives
$$
\left|
\frac{ca_1+dc_1}{ca_0+dc_0}
\right|
=
\left|
\frac{a_1}{c_0}
\right|
\cdot
\left|
\frac{1}{(a_0/c_0)-(-d/c)}
\right|
=
\left|
\frac{a_1}{c_0}
\right|
\cdot
\left|
\frac{1}{\zeta_0-(G^{-1}\cdot\infty)}
\right|
\leq
\left|
\frac{a_1}{c_0\cdot\kappa_1}
\right|,
$$
and Part (1) follows observing that $a_1/c_0$ varies in a finite set of values. Finally, if $c,c_1\not=0$ and $c_0=0$ then
$$
\left|
\frac{ca_1+dc_1}{ca_0+dc_0}
\right|
=
\left|
\frac{a_1}{a_0}-(-d/c)\frac{c_1}{a_0}
\right|
\leq
\left|\frac{a_1}{a_0}\right|+|G^{-1}\cdot\infty|\left|\frac{c_1}{a_0}\right|.
$$
In this case $\zeta_0=\infty$, which is not an endpoint of $[\widehat{b_m}]$. Thus $[\widehat{b_m}]$ is contained in the compact interval of $\RR$ delimited by the two parallel vertical segments of $\Omega_\HH$. Hence $|G^{-1}\cdot\infty|$ is uniformly bounded, because the pole $G^{-1}\cdot\infty$ belongs to the closure of $[\widehat{b_m}]$ (see proof of Lemma~\ref{LemmaDistancePoleReducedForm}). Part (1) follows in this case too, and the proof is complete. Part (2) follows similarly, replacing 
$\zeta_1$ by 
$
\zeta_\ast:=G_{b_{m+1}}\cdot\zeta_2
$ 
and observing that, since $G_{b_{m+1}}$ varies in the finite set 
$\{G_a:a\in\cA\}$ then also the entries of $X\in\sltwor$ with
$
G_{b_{m+1}}\cdot\zeta_2=X\cdot \infty
$ 
vary in a finite set. Moreover $\zeta_0\not=\zeta_\ast$, and thus 
$
G\cdot\zeta_0\not=G\cdot\zeta_\ast
$. 
\end{proof}

\subsection{Proof of Theorem~\ref{TheoremGoodApproximationsAndBoundaryExpansion}}

According to a standard separation property of parabolic fixed points, there exists a constant $S_0>0$, depending only on $\Gamma$ and on $\cS$, such that for any $G\cdot z_i$ and $F\cdot z_j$ in $\cP_\Gamma$ with 
$
G\cdot z_i\not=F\cdot z_j
$ 
we have 
\begin{equation}
\label{EquationDistanceDisjointHoroballs}
|G\cdot z_i-F\cdot z_j|
\geq
\frac{1}{S_0}\cdot\frac{1}{D(G\cdot z_i)D(F\cdot z_j)}.
\end{equation}
Se Appendix \S~A in~\cite{MarcheseDimension} for a proof. 
Let 
$
\alpha=[a_0,a_1,\dots]_\HH=[W_0,W_1,\dots]_\HH
$ 
be the expansion of $\alpha\in\RR$ as in 
Equation~\eqref{EquationCuspidalCuttingSequence(UpperHalfPlane)}.

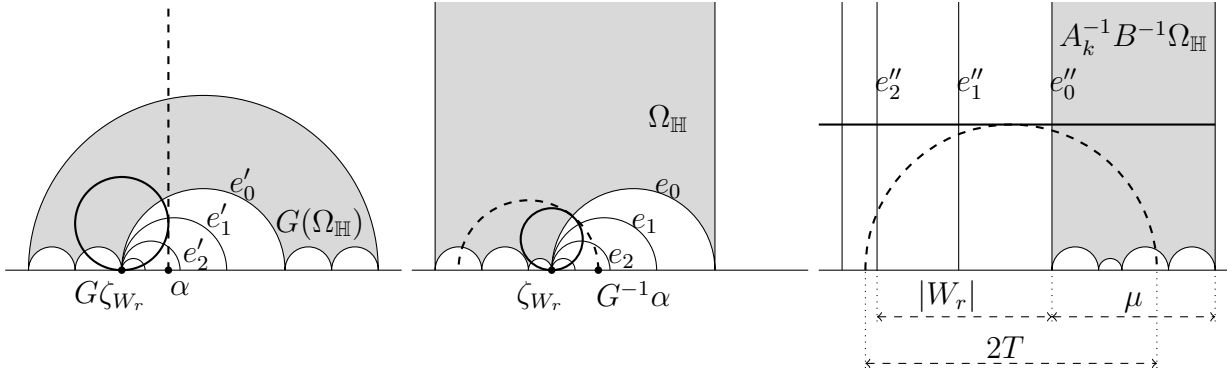
\begin{figure}
\begin{center}
\begin{tikzpicture}[scale=0.31]

\clip(-1,-5) rectangle (16,11.5);

\draw[thin] (-1,0) -- (16,0);

\filldraw[fill=black!15!white, draw=black] 
(15,0) arc (0:180:1) 
arc  (0:180:1)
arc (0:180:3.5) 
arc (0:180:1) 
arc (0:180:1) 
arc (180:0:7.5);

\node[] at (12.5,2) {$G(\Omega_\HH)$};

\draw  (4,0) arc (180:0:0.5);
\node[] at (7.2,0.7) {$e'_2$};
\draw  (4,0) arc (180:0:1.25);
\node[] at (8.2,2.2) {$e'_1$};
\draw  (4,0) arc (180:0:2.25);
\node[] at (9.2,3.7) {$e'_0$};

\draw[thick,dashed]  (6,0) -- (6,12);
\node[circle,fill,inner sep=1pt] at (6,0) {};
\node[] at (6.5,0) [below] {$\alpha$};

\node[circle,fill,inner sep=1pt] at (4,0) {};
\node[] at (3.5,0) [below] {$G\zeta_{W_r}$};

\draw[thick] (4,2) circle (2);

\end{tikzpicture}
\begin{tikzpicture}[scale=0.31]

\clip(-1,-5) rectangle (16,11.5);

\draw[thin] (-1,0) -- (16,0);


\filldraw[fill=black!15!white, draw=black] 
(12,0) arc (0:180:3.5) 
arc  (0:180:0.5)
arc (0:180:1) 
arc (0:180:1) -- (0,16.5) -- (12,16.5) -- (12,0);
\node[] at (10,6.5) {$\Omega_\HH$};

\draw  (5,0) arc (180:0:0.5);
\node[] at (8,0.5) {$e_2$};
\draw  (5,0) arc (180:0:1.25);
\node[] at (9,2) {$e_1$};
\draw  (5,0) arc (180:0:2.25);
\node[] at (10,3.5) {$e_0$};

\draw[thick,dashed]  (7,0) arc (0:180:3);
\node[circle,fill,inner sep=1pt] at (7,0) {};
\node[] at (8.5,0) [below] {$G^{-1}\alpha$};

\node[circle,fill,inner sep=1pt] at (5,0) {};
\node[] at (4.5,0) [below] {$\zeta_{W_r}$};

\draw[thick] (5,1.33) circle (1.33);

\end{tikzpicture}
\begin{tikzpicture}[scale=0.31]

\clip(-1,-5) rectangle (16.5,11.5);

\draw[thin] (-1,0) -- (17,0);

\draw[<->,dashed] (1.5,-2) -- (9,-2);
\node[] at (4.5,-2.25) [above] {$|W_r|$};
\draw[thin,dotted] (1.5,-2.2) -- (1.5,0);
\draw[thin,dotted] (9,-2.2) -- (9,0);

\draw[<->,dashed] (9,-2) -- (16,-2);
\node[] at (12.5,-2.25) [above] {$\mu$};
\draw[thin,dotted] (16,-2.2) -- (16,0);

\filldraw[fill=black!15!white,draw=black] 
(16,12) -- (16,0)
arc (0:180:1) 
arc  (0:180:1)
arc (0:180:0.5) 
arc (0:180:1) -- (9,12) -- (16,12);


\draw (0,0) -- (0,12);
\node[] at (2,8) {$e''_2$};
\draw (1.5,0) -- (1.5,12);
\node[] at (5.5,8) {$e''_1$};
\draw (5,0) -- (5,12);
\node[] at (9.5,8) {$e''_0$};

\draw[thick,dashed] (1,0) arc (180:0:6.25);

\draw[thick] (-1,6.25) -- (16,6.25);

\draw[<->,dashed] (1,-4) -- (13.5,-4);
\node[] at (7,-4.25) [above] {$2T$};
\draw[thin,dotted] (1,-4.2) -- (1,0);
\draw[thin,dotted] (13.5,-4.2) -- (13.5,0);

\node[] at (12.5,10) {$A_k^{-1}B^{-1}\Omega_\HH$};

\end{tikzpicture}
\end{center}
\caption{The $r$-th cuspidal word $W_r=(a_0,a_1,a_2)$ of $\alpha$ is the first cuspidal word of $G^{-1}\cdot\alpha$, where $G=G_{W_0,\dots,W_{r-1}}$. The vertex 
$\zeta_{W_r}$ of $\Omega_\HH$ is common to the arcs $e_0=e_{a_0}$, $e_1:=G_{a_0}e_{a_1}$ and $e_2:=G_{a_0a_1}e_{a_2}$. The arcs $e'_i=Ge_i$ share the vertex $G\zeta_{W_r}$. The point $\zeta_{W_r}$ is sent to 
$\infty$, and the arcs $e_0,e_1,e_2$ are sent to the parallel vertical arcs $e''_0,e''_1,e''_2$. 
We have $|W_r|=\big|\re(e''_2)-\re(e''_0)\big|$.}
\label{FigureProofLemmaConvergentsAndGeometricLength}
\end{figure}

\subsubsection{First part}

Fix $r\in\NN$ with $|W_r|>0$. Take $k\in\{1,\dots,p\}$ and $B\in\Gamma$ as in 
Equation~\eqref{EquationVerticesDomainInfinity}, that is $\zeta_{W_r}=BA_k\cdot \infty$. 
As in Figure~\ref{FigureProofLemmaConvergentsAndGeometricLength}, let $T>0$ be such that the horoball
$$
B_T:=G_{W_0,\dots,W_{r-1}}BA_k\big(\{z\in\HH:\im(z)>T\}\big)
$$
is tangent at $G_{W_0,\dots,W_{r-1}}\cdot \zeta_{W_r}$ with radius 
$
\rho(B_T)=|\alpha-G_{W_0,\dots,W_{r-1}}\cdot \zeta_{W_r}|
$. 
Equation~\eqref{EquationDiameterHoroball} and 
Equation~\eqref{EquationDefinitionDenominator} give
$$
D(G_{W_0,\dots,W_{r-1}}\cdot\zeta_{W_r})^2
\cdot
|\alpha-G_{W_0,\dots,W_{r-1}}\cdot\zeta_{W_r}|
=
c^2(G_{W_0,\dots,W_{r-1}}BA_k)\cdot\frac{\diameter(B_T)}{2}
=
\frac{1}{2T}.
$$
The geodesic in $\HH$ with endpoints 
$(G_{W_0,\dots,W_{r-1}}BA_k)^{-1}\cdot \infty$ and 
$(G_{W_0,\dots,W_{r-1}}BA_k)^{-1}\cdot \alpha$ is tangent to $\{z\in\HH:\im(z)>T\}$. 
Finally Equation~\eqref{EquationDefinitionGeometricLenght} gives
$$
|W_r|\leq 2T\leq |W_r|+2\mu.
$$

\subsubsection{Second part}

Referring to \S~\ref{SectionParabolicFixedPointsAndFiniteIndexFreeSubgroup}, let $\zeta_0$ be the vertex of $\Omega_\HH$ and $(b_0,\dots,b_m)$ be the admissible word such that the reduced form of the parabolic fixed point $G\cdot z_k$ is
$$
G\cdot z_k=G_{b_0,\dots,b_m}\cdot\zeta_0,
$$
where $\zeta_0$ is not an endpoint of $e_{\widehat{b_m}}$ whenever $(b_0,\dots,b_m)$ is not the empty word. Assume  
$
D(G\cdot z_k)^2|\alpha-G\cdot z_k|<\epsilon_0
$, 
where the constant $\epsilon_0>0$ will be determined later.

\smallskip

\emph{Step $(0)$} Assume that $(b_0,\dots,b_m)$ is the empty word, so that  
$
\zeta_0=G\cdot z_k\not=\infty
$. 
Consider the extra assumption $|W_0|>0$ and $\zeta_0=\zeta_{W_0}$ on pairs $(\alpha,\zeta_0)$, where 
$
\zeta_{W_0}=\varphi^{-1}(\xi_{W_0})
$ 
and $\xi_{W_0}$ is the vertex of $\Omega_\DD$ associated to $W_0$ as in \S~\ref{SectionCuspidalWords}. 
Define $\epsilon_0>0$ by 
$$
\epsilon_0:=\inf_{(\alpha,\zeta_0)}D(\zeta_0)^2\cdot|\alpha-\zeta_0|,
$$
where the infimum is taken over all pairs $(\alpha,\zeta_0)$ not satisfying the extra assumption. With such $\epsilon_0$, the statement follows whenever $(b_0,\dots,b_m)$ is the empty word.

\smallskip

\emph{Step $(1)$} Now assume that $(b_0,\dots,b_m)$ is not the empty word. Then $G\cdot z_k$ is an interior point of  
$
[b_0,\dots,b_m]_\HH
$. 
Let $\zeta_1,\zeta_2$ be the endpoints of $[\widehat{b_m}]$, which are vertices of $\Omega_\HH$ different from $\zeta_0$. The endpoints of 
$
[b_0,\dots,b_m]_\HH
$ 
are 
$
\zeta'_i:=G_{b_0,\dots,b_m}\cdot\zeta_i
$ 
for $i=1,2$, according to Equation~\eqref{EquationDefinitionCylinderBowenSeries}. 
Let $N\geq-1$ be maximal with $a_n=b_n$ for any $n=0,\dots,N$, where the last condition is empty for $N=-1$, and where 
$N\leq m$. Observe that condition $N\leq m-1$ implies  
$
\alpha\not\in [b_0,\dots,b_m]_\HH
$, 
and therefore 
\begin{align*}
|\alpha-G\cdot z_k|
&
\geq
\min_{i=1,2}
|\zeta'_i-G\cdot z_k|
=
\min_{i=1,2}
|G_{b_0,\dots,b_m}\cdot\zeta_i-G_{b_0,\dots,b_m}\cdot\zeta_0|
\\
&
\geq
\frac{S_0^{-1}}{D(G_{b_0,\dots,b_m}\cdot\zeta_0)}\cdot
\min_{i=1,2}
\frac{1}{D(G_{b_0,\dots,b_m}\cdot\zeta_i)}
\geq
\frac{S_0^{-1}\kappa_2}{D(G_{b_0,\dots,b_m}\cdot\zeta_0)^2},
\end{align*}
where the third inequality follows from 
Part (1) of Lemma~\ref{LemmaInequalityDenominatorsReducedForm} and the second from 
Equation~\eqref{EquationDistanceDisjointHoroballs}. Therefore $N=m$, provided that 
$$
\epsilon_0<\kappa_2/S_0.
$$
We proved 
$
[a_0,\dots,a_m]_\HH=[b_0,\dots,b_m]_\HH
$. 
Moreover $G\cdot z_k$ does not belong to the interior of $[a_0,\dots,a_m,a_{m+1}]_\HH$, since the latter is a subinterval of $[b_0,\dots,b_m]_\HH$ delimited by the image under $G_{b_0,\dots,b_m}$ of two consecutive vertices of $\Omega_\HH$. The same argument as in the first part of Step (1), which is left to the reader, shows that $G\cdot z_k$ is an endpoint of $[a_0,\dots,a_m,a_{m+1}]_\HH$. 

\smallskip

\emph{Step $(2)$} We show that 
$
G\cdot z_k=G_{b_0,\dots,b_m}\cdot \zeta_0
$ 
is an endpoint of $[a_0,\dots,a_{m+2}]_\HH$. Otherwise $G\cdot z_k$ doesn't belong to the closure of 
$[a_0,\dots,a_{m+2}]_\HH$. Since   
$
\alpha\in[a_0,\dots,a_{m+2}]_\HH
$  
then 
\begin{align*}
|\alpha-G\cdot z_k|
&
\geq
|G_{b_0,\dots,b_m,a_{m+1}}\cdot\zeta_3-G_{b_0,\dots,b_m}\cdot \zeta_0|
\\
&
\geq
\frac
{S_0^{-1}}
{D(G_{b_0,\dots,b_m}\cdot\zeta_0)D(G_{b_0,\dots,b_m,a_{m+1}}\cdot\zeta_3)}
\geq
\frac
{S_0^{-1}\kappa_2}
{D(G_{b_0,\dots,b_m}\cdot\zeta_0)^2},
\end{align*}
where 
$
G_{b_0,\dots,b_m,a_{m+1}}\cdot\zeta_3
$ 
is the endpoint of $[a_0,\dots,a_{m+2}]_\HH$ which is closest to $G\cdot z_k$ and where $\zeta_3$ is a vertex of $\Omega_\HH$ which is not an endpoint of $e_{\widehat{a_{m+1}}}$. We use Equation~\eqref{EquationDistanceDisjointHoroballs} and Part (2) of Lemma~\ref{LemmaInequalityDenominatorsReducedForm}. The inequality is absurd by condition 
$
\epsilon_0<\kappa_2/S_0
$. 

\smallskip

\emph{Step (3)} Let $r$ be minimal such that $(a_0,\dots,a_m)$ is an initial factor of $W_0\ast\dots\ast W_{r-1}$. If $(a_0,\dots,a_{m+2})$ is also an initial factor of $W_0\ast\dots\ast W_{r-1}$, then
$
G_{W_0,\dots,W_{r-1}}\cdot \xi_{W_{r-1}}
$ 
is a common endpoint of the intervals $[a_0,\dots,a_m]_\HH$, $[a_0,\dots,a_{m+1}]_\HH$ and 
$[a_0,\dots,a_{m+2}]_\HH$, according to Equation~\eqref{EquationCommonEndpointsCuspidalWords}. 
Without loss of generality we have 
$$
G_{W_0,\dots,W_{r-1}}\cdot \xi_{W_{r-1}}
=
\inf[a_0,\dots,a_m]_\HH
=
\inf[a_0,\dots,a_{m+1}]_\HH
=
\inf[a_0,\dots,a_{m+2}]_\HH.
$$
The common endpoint is not $G\cdot z_k$, which belongs to the interior of 
$[a_0,\dots,a_m]_\HH$. Thus Step $(1)$ implies   
$
G\cdot z_k=\sup[a_0,\dots,a_{m+1}]_\HH
$, 
which is absurd because $G\cdot z_k$ is an endpoint of $[a_0,\dots,a_{m+2}]_\HH$ by Step $(2)$. Hence 
$
W_0\ast\dots\ast W_{r-1}
$ 
is either equal to $(a_0,\dots,a_m)$ or to $(a_0,\dots,a_{m+1})$. Moreover $(a_{m+1},a_{m+2})$ is a cuspidal word, because $[a_0,\dots,a_{m+1}]_\HH$ and $[a_0,\dots,a_{m+2}]_\HH$ share the endpoint 
$G\cdot z_k$.

\smallskip

$\bullet$ In case  
$
W_0\ast\dots\ast W_{r-1}=(a_0,\dots,a_m)
$ 
the word $(a_{m+1},a_{m+2})$ is an initial factor of $W_r$, that is $|W_r|>0$ and $\zeta_0=\zeta_{W_r}$. 

\smallskip

$\bullet$ In case    
$
W_0\ast\dots\ast W_{r-1}=(a_0,\dots,a_{m+1})
$ 
the word $W':=(a_{m+1})\ast W_r$ is also cuspidal (this is allowed by 
Remark~\ref{RemarkConcatenationCuspidalWords}). 
If $|W_r|=0$, that is $W_r=(a_{m+2})$, then $G\cdot z_k$ does not belong to the closure of 
$[a_0,\dots,a_{m+3}]_\HH$ and we get an absurd by \footnote{We reason as in Step $(2)$, modulo replacing the constant $\kappa_2$ by a smaller one, and extending Part (2) of 
Lemma~\ref{LemmaInequalityDenominatorsReducedForm} one more step, in order to compare 
$D(G_{b_0,\dots,b_m}\cdot\zeta_0)$ and 
$D(G_{b_0,\dots,b_m,a_{m+1},a_{m+2}}\cdot\zeta_3)$.}
$$
|\alpha-G\cdot z_k|
\geq
|G_{b_0,\dots,b_m}\cdot \zeta_0-G_{b_0,\dots,b_m,a_{m+1},a_{m+2}}\cdot\zeta_3|
\geq
\frac
{S_0^{-1}\kappa_2}
{D(G_{b_0,\dots,b_m}\cdot\zeta_0)^2},
$$
where $\zeta_3$ is a vertex of $\Omega_\HH$ and 
$
G_{b_0,\dots,b_m,a_{m+1},a_{m+2}}\cdot\zeta_3
$ 
is the endpoint of $[a_0,\dots,a_{m+3}]_\HH$ which is closest to $G\cdot z_k$. 
Since $W'$ is cuspidal with $|W'|>0$ we have 
$\zeta_0=\zeta_{W'}$. But we have also 
$
\zeta_{W'}=G_{a_{m+1}}\cdot\zeta_{W_r}
$, 
which implies 
$$
G_{b_0,\dots,b_m}\cdot\zeta_0=
G_{a_0,\dots,a_m}\cdot G_{a_{m+1}}\cdot\zeta_{W_r}=
G_{W_0,\dots,W_{r-1}}\cdot\zeta_{W_r}.
$$
In both cases Theorem~\ref{TheoremGoodApproximationsAndBoundaryExpansion} is proved. $\qed$

\appendix

\end{document}